\documentclass{amsart}
\usepackage{amsmath,amsthm,amsfonts,amssymb,amscd}
\usepackage[all]{xy}
\usepackage[english]{babel}
\usepackage{graphicx,color}
\usepackage[T1]{fontenc}
\usepackage[latin1]{inputenc}
\usepackage{ae}

\usepackage[mathscr]{eucal}
\usepackage{amsmath,amssymb}
\usepackage{graphics}
\usepackage{multirow}
\usepackage{hyperref}
\usepackage[active]{srcltx}

\usepackage{vmargin}

\setpapersize{A4}
\setmargins{3.0cm}       
{1.5cm}                        
{13.5cm}                      
{23.42cm}                    
{10pt}                           
{1cm}                           
{0pt}                             
{2cm}

\newtheorem{thm}{Theorem}[section]

\newtheorem{lemma}[thm]{Lemma}

\theoremstyle{definition}

\newtheorem{remark}[thm]{Remark}

\newcommand{\nc}{\newcommand}
\nc {\hh}{\check{h}}
\nc {\DD}{\mathcal{D}}
\nc {\RR}{\mathcal{R}}
\nc {\Pp}{\mathbb{P}}
\nc {\Ss}{\mathcal{S}}
\nc {\PP}{\mathbb{P}^{2}}
\nc {\Pd}{ \check{\mathbb{P}}^{2}}
\nc {\WW}{\mathcal{W}}
\nc {\Sym}{\mathrm{Sym}}
\nc {\OO}{\mathcal{O}}
\nc {\CC}{\mathbb{C}}
\nc {\EE}{\mathcal{E}}
\nc {\MM}{\mathcal{M}}
\nc {\KK}{\mathcal{K}}
\nc {\PW}{\mathcal{P}}
\nc {\NW}{\mathcal{N}_{\WW}}
\nc {\FF}{\mathcal{F}}
\nc {\GG}{\mathcal{G}}
\nc {\ZZ}{\mathcal{Z}}
\nc {\LL}{\mathcal{L}}
\nc {\HH}{\mathcal{H}}
\nc {\NN}{\mathcal{N}}
\nc {\VV}{\mathcal{V}}
\nc {\Ww}{\mathbb{W}}
\nc {\QQ}{\mathbb{Q}}
\nc {\II}{\mathcal{I}}
\nc {\tang}{\mathrm{Tang}}

\date{\today}

\begin{document}

\title{Automorphisms of projective structures}
\author[M. Falla Luza, F. Loray]
{Maycol Falla Luza$^{1}$, Frank Loray$^2$}
\address{\newline $1$ UFF, Universidad Federal Fluminense, Rua Prof. Marcos Waldemar de Freitas Reis, S/N -- Bloco H, 4o Andar
Campus do Gragoat\'a,
Niter\'oi, RJ, Brazil, Brasil\hfill\break
$2$ Univ Rennes, CNRS, IRMAR, UMR 6625, F-35000 Rennes, France} 
\email{$^1$ hfalla@id.uff.br} \email{$^2$ frank.loray@univ-rennes1.fr}
\thanks{The authors thank Brazilian-French Network in Mathematics and CAPES/COFECUB Project Ma 932/19 ``Feuilletages holomorphes et int\'eractions avec la g\'eom\'etrie''. 
Falla Luza acknowledges support from CNPq (Grant number 402936/2021-3).
Loray is supported by CNRS and  Centre Henri Lebesgue, program ANR-11-LABX-0020-0.}

\begin{abstract}
We study the problem of classifying local projective structures in dimension two having non trivial Lie symmetries. In particular we obtain a classification of flat projective structures having positive dimensional Lie algebra of projective vector fields.
\end{abstract}

\maketitle

\section{Introduction}\label{sec:Automorphismes}

Let $\Pi$ be a projective structure defined on some neighborhood $U$ of $0\in\mathbb C^2$ by
\begin{equation}\label{projective-structure}
y''=f(x,y,y')
\end{equation}
$$\text{with}\ \ \ f(x,y,z)=A(x,y)+B(x,y)z+C(x,y)z^2+D(x,y)z^3$$
The projectivized tangent bundle $M=P(TU)$ is naturally a contact manifold and each solution on $U$ lifts uniquely as a Legendrian curve on $M$ defining a foliation $\mathcal G$ on $M$. If $U$ is a small euclidean ball, the space of $\mathcal{G}$--leaves is a complex surface $U^*$ which contains a rational curve $C_0$ of self-intersection $+1$ given by the image of $\mathbb{P}(T_0 U)$ (equivalently, $C_0$ corresponds to the solutions of \ref{projective-structure} passing through $0\in U$). We call the pair $(U^*, C_0)$ the dual neighborhood of $(U,0)$. For a survey of such duality, we refer to \cite{FL}.

A local diffeomorphism $\Psi$ in $U$ (fixing $0$ or not) is an automorphism of the projective structure $\Pi$,
if it sends geodesics to geodesics of (\ref{projective-structure}). As a consequence, $\Psi$ acts also on the 
dual neighborhood $(U^*,C_0)$, inducing a diffeomorphism $\check{\Psi}$ from the neighborhood of $C_0$
onto the neighborhood of (itself or another) $(+1)$-rational curve $C$ inside $U^*$. Conversely, such a diffeomorphism
$\check{\Psi}$ in $U^*$, between neighborhoods of $(+1)$-rational curves, induces an automorphism $\Psi$
of the projective structure $\Pi$ as above. 
Lie showed that the pseudo-group of automorphisms of the projective structure forms a Lie pseudo-group 
denoted by $\mathrm{Aut}(\Pi)$.
Vector fields whose local flow belong to this pseudo-group are called infinitesimal symmetries 
and form a Lie algebra denoted by $\mathfrak{aut}(\Pi)$. Elements of $\mathfrak{aut}(\Pi)$ 
obviously correspond to germs of holomorphic vector fields on the dual (germ of) neighborhood $(U^*, C_0)$,
and we denote by $\mathfrak{aut}(U^*, C_0)$ the corresponding Lie algebra. Clearly, we have 
$$\mathfrak{aut}(\Pi)\simeq\mathfrak{aut}(U^*, C_0).$$
In \cite{Lie}, Lie gives a classification of the possible infinitesimal symmetry algebras for projective structures, 
showing that they must be isomorphic to one of the following algebras
\begin{equation}\label{LieClassSymProjStr}
\{0\},\ \ \ \mathbb{C},\ \ \ \mathrm{aff}(\mathbb{C}),\ \ \ \mathrm{sl}_2(\mathbb{C})\ \ \ \text{or}\ \ \ \mathrm{sl}_3(\mathbb{C})
\end{equation}
where $\mathrm{aff}(\mathbb{C})$ is the non commutative $2$-dimensional Lie algebra 
corresponding to the affine group acting on the line:
it is spanned by $X$ and $Y$ satisfying $[X,Y]=X$.

In their paper \cite{BMM}, R. Bryant, G. Manno and V. Matveev classified two-dimensional local metrics $(U,g)$
whose underlying projective structure $(U,\Pi_g)$ is such that $\dim \mathfrak{aut}(\Pi_g)=2$. In \cite{BDE} a local obstruction to the existence of a Levi-Civita connection within a given projective structure on a surface is given. More recently, in \cite{Matveev} the case $\dim \mathfrak{aut}(\Pi_g)=1$ is solved. These problems were settled
by Lie himself. As a biproduct, they provide in \cite[section 2.3]{BMM} a list of almost unique normal forms 
for generic local projective structures $(U,\Pi)$ with $\dim \mathfrak{aut}(\Pi)=2$. 
In section \ref{sec:LieSym}, we give a precise statement of this, completed with other possible dimensions $\dim\mathfrak{aut}(\Pi)=1,2,3,8$. We mention here the recent paper \cite{DW} where the authors establish an explicit correspondence between two-dimensional projective structures admitting a projective vector field and a class of solutions to the $SU(\infty)$
 Toda equation. There, our preliminary version of the current work \cite[section 6]{FL0} is used by the authors.

Then we focus on flat projective structures. These are structures whose solutions are given by a pencil of (transverse) foliations or, equivalently whose dual surface $(U^*, C_0)$ admits a semi-local fibration transverse to $C_0$. A non-linearizable projective structure admits at most two flat structures, see \cite[Theorem 5.1]{FL} or \cite[Theorem A]{FL1}. In other words, if $(U,\Pi)$ admits at least three flat structures then there is a local biholomorphism sending solutions of $\Pi$ in solutions of the linear structure $y''=0$. We use this result in order to prove our main result, Theorem \ref{thm:flat+LieSym} where we  classify flat projective structures admiting positive dimensional Lie algebra of symmetries, see Section \ref{sec:SymFlat}.

\subsection*{Organization of the paper} 
In Section $2$ we present some preliminar properties about projective structures with one vector field of symmetries and webs with non trivial automorphism Lie algebra. In Section $3$ we establish a normal form of projective structures having non trivial Lie symmeties, see Theorem \ref{thm:Bryant&al}. Finally in Section $4$ we present our main result, Theorem \ref{thm:flat+LieSym} classifying flat projective structures having Lie symmetries. A preliminary version of this work was posted in \cite[section 6]{FL0}.
\section{Preliminaries}

Let us start with some considerations in the case the projective structure $(U,\Pi)$ is invariant
by one (regular) vector field, say $\partial_y$.

\begin{lemma}\label{lem:XdualSingular}
Let $X=\partial_y$ be a non trivial symmetry of a projective structure $(U,\Pi)$
and let $\check{X}$ be the dual vector field on $(U^*,C_0)$. Then the differential
equation for the projective structure takes the form
\begin{equation}\label{eq:normform1vectfield}
y''=A(x)+B(x)(y')+C(x)(y')^2+D(x)(y')^3
\end{equation}
and we have the following possibilities:
\begin{itemize}
\item $D(0)\not=0$ and $\check{X}$ is regular on $(U^*,C_0)$, with exactly one tangency
with $C_0$;
\item $D(0)=0$ but $D\not\equiv0$ and $\check{X}$ has an isolated singular point on $(U^*,C_0)$;
\item $D\equiv0$ and $\check{X}$ has a curve $\Gamma$ of singular points on $(U^*,C_0)$;
moreover, $\Gamma$ is transversal to $C_0$ and the saturated foliation $\mathcal F_{\check{X}}$
defines a fibration transversal to $C_0$. 
\end{itemize}
\end{lemma}

\begin{proof}The normal form (\ref{eq:normform1vectfield}) follows from a straithforward computation.
Clearly, $\check{X}$ has a singular point at $p\in U^*$  if, and only if, the corresponding geodesic in $U$ 
is a trajectory of $X$ (i.e. is $X$-invariant). Therefore, up to shrinking the neighborhoods $U$ and $U^*$,
we have $3$ possibilities: 
\begin{itemize}
\item $D\not=0$ on $U$ and $\check{X}$ is regular on $(U^*,C_0)$;
\item $D$ vanishes exactly along $x=0$ which is therefore geodesic, 
and $\check{X}$ has an isolated singular point at the corresponding point on $(U^*,C_0)$;
\item $D\equiv0$, the foliation $dx=0$ defined by $X$ is geodesic, 
and $\check{X}$ has a curve $\Gamma$ of singular points on $(U^*,C_0)$. 
\end{itemize}
In the first case, the restriction $\check{X}\vert_{C_0}$ cannot be identically tangent 
to $TC_0\simeq\mathcal O_{\mathbb P^1}(2)$,
otherwise it would have a singular point; it thus defines a non trivial section of $NC_0\simeq\mathcal O_{\mathbb P^1}(1)$ 
which must have a single zero, meaning that $\check{X}$ has a single tangency with $C_0$. The second case we are not interested, since moving to a nearby point of $U$, we can assume that we are 
in the first case. 

In the third case, each trajectory of $X$ is geodesic, so is the induced foliation $dx=0$.
By duality, this foliation defines a cross section to $C_0$ consisting of singular points of $\check{X}$.
The saturated foliation $\mathcal F_{\check{X}}$ is locally defined by $\check{X}$ oustide $\Gamma$, 
and by the vector field $\frac{1}{f}\check{X}$ near $\Gamma$,  where $f$ is a reduced equation of $\Gamma$.
But $\frac{1}{f}\check{X}$ induces a non zero section of $NC_0$ (near $\Gamma$) since it must be less vanishing;
$\frac{1}{f}\check{X}$ is therefore transversal to $C_0$, and so is $\mathcal F_{\check{X}}$.
\end{proof}

\begin{lemma}[Cartan {\cite[p.78-83]{CartanTissus}}]\label{lem:CartanSymTissus}
Let $\WW=\FF_0 \boxtimes \FF_1 \boxtimes \FF_{\infty}$ be a regular $3$-web on $(\CC,0)$,
and let $\mathfrak{aut}(\WW)$ be the Lie algebra of vector fields whose flow preserve $\WW$.
If $\mathfrak{aut}(\WW)\not=0$ then we are in one of the two cases, up to change of coordinates:
\begin{itemize}
\item $\mathfrak{aut}(\WW)=\CC\partial_y$ and $\WW=dy\boxtimes (dy-dx)\boxtimes (dy+f(x)(dy-dx))$, 
with $f$ analytic, not of the form $f(x)=ae^{bx}$, $a,b\in\CC$;
\item $\mathfrak{aut}(\WW)=\CC\langle\partial_x,\partial_y,x\partial_x+y\partial_y\rangle$ 
and $\WW=dy\boxtimes (dy-dx)\boxtimes dx$.
\end{itemize}
\end{lemma}

\begin{lemma}\label{lem:XgeodFlat}
Under assumptions and notations of Lemma \ref{lem:XdualSingular}, 
assume that we are in the last case $D\equiv0$.
If the singular set $\Gamma$ of $\check{X}$ is a fiber of the fibration defined by $\mathcal F_{\check{X}}$ on $(U^*,C_0)$,
then $(U,\Pi)$ is linearizable. 
\end{lemma}

\begin{proof}Take $3$ different fibers of $\mathcal F_{\check{X}}$ different from $\Gamma$, 
they define a $3$-web $\WW=\FF_0 \boxtimes \FF_1 \boxtimes \FF_{\infty}$ which is invariant by $X$.
By Cartan's Lemma \ref{lem:CartanSymTissus}, we can assume $X=\partial_y$ and 
$\WW=dy\boxtimes (dy-dx)\boxtimes (dy+f(x)(dy-dx))$ so that the flat structure of $\Pi$
is defined by the pencil $dy+zf(x)(dy-dx)=0$. But the foliation $dx=0$ defined by $X$ is, 
by assumption, belonging to the pencil, which means that $\Pi$ is also defined by the 
hexagonal $3$-web $dy\boxtimes (dy-dx)\boxtimes dx$, thus linearizable.
\end{proof}

\section{Classification of projective structures with Lie symmetries}\label{sec:LieSym}

The problem of this section is to classify those local projective structures $(U,\Pi)$
having non trivial Lie symmetry, i.e. such that $\dim \mathfrak{aut}(\Pi)>0$, up to change of coordinates.
However, in this full generality, the problem is out of reach; indeed, it includes for instance the problem
of classification of germs of holomorphic vector fields (with arbitrary complicated singular points),
which is still challenging. Instead of this, and in the spirit of Lie's work, we produce a list of possible normal
forms up to change of coordinates for such a $(U,\Pi)$ at a generic point $p\in U$,
i.e. outside a closed analytic subset consisting of singular features. For instance, a non trivial vector field
is regular at a generic point and can be rectified to $\partial_y$; we only consider this constant vector field
in the case $\dim\mathfrak{aut}(\Pi)=1$. The following resumes some results of \cite[section 2.3]{BMM}.

\begin{thm}\label{thm:Bryant&al}
Let $(U,\Pi)$ be a projective structure with $\mathfrak{aut}(\Pi)\not=\{0\}$. Then, at the neighborhood
of a generic point $p\in U$, the pair $(\Pi,\mathfrak{aut}(\Pi))$ can be reduced by local change of coordinate 
to one of the following normal forms:
\begin{itemize}
\item[(i)] $\mathfrak{aut}(\Pi)=\CC\cdot\partial_y$ and $(A,B,C,D)=$
\begin{itemize}
\item[(i.a)] $(A(x),B(x),0,1)$ with $A,B\in\CC\{x\}$;
\item[(i.b)] $(A(x),0,e^{x},0)$ with $A\in\CC\{x\}$;
\end{itemize}
\item[(ii)] $\mathfrak{aut}(\Pi)=\CC\langle\partial_y,\partial_x+y\partial_y\rangle$ and $(A,B,C,D)=$
\begin{itemize}
\item[(ii.a)] $(\alpha e^{x},\beta,0,e^{-2x})$ with $\alpha,\beta\in\CC$, $(\alpha,\beta)\not=(0,2),(0,\frac{1}{2})$;
\item[(ii.b)] $(\alpha e^x,0,e^{-x},0)$ with $\alpha\in\CC$;
\end{itemize}
\item[(iii)] $\mathfrak{aut}(\Pi)=\CC\langle\partial_y,\partial_x+y\partial_y,y\partial_x+\frac{y^2}{2}\partial_y\rangle$
and $(A,B,C,D)=(0,\frac{1}{2},0,e^{-2x})$;
\item[(iv)] $\mathfrak{aut}(\Pi)=\mathrm{sl}_3(\mathbb{C})$ and $(A,B,C,D)=(0,0,0,0)$.
\end{itemize}
These normal forms are unique, except for the case {\rm(i.a)}, which is unique up to the $\mathbb{C}^*$-action:
$$(A(x), B(x), 0, 1)\stackrel{\lambda\in\mathbb C^*}{\longrightarrow}(\lambda^3A(\lambda^2 x), \lambda^2 B(\lambda^2 x), 0, 1).$$
\end{thm}

\begin{remark}
The normal forms for $\mathfrak{aut}(\Pi)$ in the statement correspond to the list of transitive local actions 
of Lie algebras listed in (\ref{LieClassSymProjStr}), except that $\mathrm{sl}_2(\mathbb{C})$ has also exotic
representations generated by  
$$X=\partial_y,\ \ \ Y=\partial_x+y\partial_y\ \ \ \text{and}\ \ \ 
Z=(y+c_1e^x)\partial_x+(\frac{y^2}{2}+c_2e^{2x})\partial_y,\ \ \ c_1,c_2\in\CC.$$
Only the standard one occurs as symmetry algebra of a projective structure.
\end{remark}

\begin{remark}Case (iii) corresponds to the special structure 
$$\Pi_0\ :\ y'' = (x y' - y)^3$$
at the neighborhood of any point $p\not=(0,0)$ and is invariant under the standard action of $\mathrm{sl}_2(\mathbb{C})$
$$\mathfrak{aut}(\Pi_0)=\CC\langle x\partial_y,\frac{1}{2}(-x\partial_x+y\partial_y),-\frac{1}{2}y\partial_x\rangle$$ (see \cite[Theorem 3]{Rom}).
However, at $p=(0,0)$, the Lie algebra is singular, which is excluded from the list of Theorem \ref{thm:Bryant&al}.
Case (iv) corresponds to the linearizable case $y''=0$.
\end{remark}


\begin{proof}[Proof of Theorem \ref{thm:Bryant&al}]
We essentially follow \cite[section 2.3]{BMM}.
Let us start with the case $\mathfrak{aut}(\Pi)=\CC\langle X\rangle$. 
At a generic point $p\in U$, the vector field $X$ is regular and we can choose local coordinates such that $X=\partial_y$.
One easily deduce that the equation (\ref{projective-structure}) for the projective structure, being $X$-invariant, 
takes the form $y''=f(x,y,y')$ with
\begin{equation}\label{equation-with-one-vector-field}
f(x,y,z)= A(x) + B(x)z + C(x)z^2 + D(x)z^3.
\end{equation}
The normalizing coordinates for $X$ are unique up to a change of the form 
$$\Phi:(x,y)\mapsto(\psi(x),y+\phi(x)),\ \ \ \psi(0)=\phi(0)=0,\ \psi'(0)\not=0.$$
The projective structure $\Phi^*\Pi$ is defined by 
$$f(x,y,z)=\widehat{A}(x) + \widehat{B}(x)z +\widehat{C}(x)(z)^2 + \widehat{D}(x) (z)^3$$
where (we decompose for simplicity)
\begin{equation}\label{primeiro-cambio} 
\text{when}\ \Phi=(\psi(x),y),\ \text{then} 
\left\{\begin{matrix}
\widehat A=A\circ \psi\cdot(\psi')^2\hfill\\
\widehat B=B\circ \psi\cdot\psi'+\frac{\psi''}{\psi'}\\
\widehat C=C \circ \psi\hfill\\
\widehat D=\frac{D\circ \psi}{ \psi'}\hfill
\end{matrix}\right.
\end{equation}

\begin{equation}\label{segundo-cambio}  \text{when}\ \Phi=(x, y + \phi(x)),\ \text{then} 
\left\{\begin{matrix}
\widehat A=A+ B\phi'  + C(\phi')^2 +D (\phi')^3  - \phi''\\
\widehat B=B + 2C \phi'  + 3D (\phi')^2\hfill\\
\widehat C=C+ 3D\phi'\hfill\\
\widehat D=D\hfill
\end{matrix}\right.
\end{equation}
If $D\not\equiv 0$, then we can assume at a generic point that $D\neq 0$. We can normalize $\widehat{D}=1$ 
by setting $\psi^{-1}(x):=\int_0^x \frac{d\zeta}{D(\zeta)}$ in the first change,
and then normalize $\widehat{C}=0$ by setting $\phi'(x)=-C/3$ (which does not change $D=1$): $(A(x), B(x), 0,1)$.

Since we are interested in the Lie algebra, more than a given vector field, then we can also change 
$\Phi(x,y)=(x,ay)$ with $a\not=0$ and get the form
\begin{equation}\label{terceiro-cambio} 
\text{when}\ \Phi=(x,ay),\ \text{then}\ (\widehat A,\widehat B,\widehat C, \widehat D)=(a^{-1}A,\hspace{0.2cm} B,\hspace{0.2cm} aC,\hspace{0.2cm} a^2D).
\end{equation}
Finally, $\Phi=(a^2x,ay)$, a combination of (\ref{primeiro-cambio}) and (\ref{terceiro-cambio}), yields the new normal form 
$$(a^3A(a^2 x), a^2 B(a^2 x), 0, 1)$$
whence the $\CC^*$-action of the statement.

Suppose now $D\equiv 0$. If $C$ would be constant then, by (\ref{eq:Liouville}), $L_1 = L_2 =0$ and $\Pi$ is linearizable. 
Passing to a generic point, we can assume $C'(0)\neq 0$ and use changes (\ref{primeiro-cambio}) and (\ref{terceiro-cambio}) to normalize $C=e^x$. Finally by using (\ref{segundo-cambio}), we arrive in the unique desired normal form $(A(x),0, e^x,0)$. In this case the equation is never linearizable, since by (\ref{eq:Liouville}) we get $(L_1,L_2)=(0,2e^{2x})$.

Now we study the case $\mathfrak{aut}(\Pi)=\CC\langle X, Y\rangle$, with $[X,Y]=X$. By \cite[Lemma 1]{BMM}, 
we know that, at a generic point, we can find coordinates where $X=\partial_y$, $Y= \partial_x + y\partial_y$. 
The invariance of the projective structure by both the flows of $X$ and $Y$ yields
$$(A,B,C,D)=(\alpha e^x, \beta, \gamma e^{-x}, \delta e^{-2x}),$$
were $\alpha$, $\beta$, $\gamma$ and $\delta$ are constants. 
The normalizing coordinates for the Lie algebra are unique up to a change
$$\Phi=(x,ay+be^x)$$
with $a,b\in\CC$, $a\not=0$. 
If $\delta\neq 0$, we obtain a unique normal form $(\alpha e^x, \beta , 0, e^{-2x})$. 
By \cite[Lemma 4]{BMM}, the cases $(\alpha,\beta)=(0,2)$ and $(0,\frac{1}{2})$ have more symmetries:
they respectively correspond to the $\mathrm{sl}_3(\mathbb{C})$ and $\mathrm{sl}_2(\mathbb{C})$ cases.

When $\delta=0$, we shall have $\gamma\neq 0$ (otherwise $\Pi$ would be linearizable), 
and we can normalize $(\widehat A,\widehat B,\widehat C, \widehat D)=(\alpha e^x, 0, e^{-x}, 0)$, with $\alpha\in \mathbb{C}$;
this normal form is unique. 

The case $\mathfrak{aut}(\Pi)=\mathrm{sl}_2(\mathbb{C})$ follows directly from \cite[Lemma 4]{BMM}. 
\end{proof}

In Theorem \ref{thm:Bryant&al}, normal forms (i) contain some models with larger symmetry Lie algebra,
and we end the section by determining them. 

\subsection*{In the case  $\mathfrak{aut}(\Pi)=\mathrm{sl}_3(\mathbb{C})$} First of all, we have from \cite{Liouville} that the projective structure is linearizable when Liouville invariants $L_1$ and $L_2$ given by 
\begin{equation}\label{eq:Liouville}
\left\{\begin{matrix}
L_1=2B_{xy}-C_{xx} -3A_{yy}-6AD_x -3A_xD +3(AC)_y  +BC_x - 2BB_y,\\
L_2=2C_{xy}-B_{yy} -3D_{xx}+6A_yD +3AD_y  -3(BD)_x -B_yC + 2CC_x.
\end{matrix}\right.
\end{equation}
are identically zero, and we get:
$$ (L_1,L_2)=(-3A'(x),-3B'(x))\ \text{for model (i.a), and } $$
$$(L_1,L_2)=(-e^{-x},-2e^{-2x})\ \text{for model (i.b).}$$
Linearizability only occur in model (i.a) when $A$ and $B$ are simultaneously constant.

\subsection*{In the case $\mathfrak{aut}(\Pi)=\mathrm{aff}(\mathbb{C})$} There must exists a vector field $v\in\mathfrak{aut}(\Pi)$ such that
$$[\partial_y,v]=\partial_y\ \ \ \text{or}\ \ \ c\cdot v,\ \ \ c\in\CC.$$
This implies that $v$ takes the respective form
$$v=\alpha(x)\partial_x+(y+\beta(x))\partial_y\ \ \ \text{or}\ \ \ e^{cy}(\alpha(x)\partial_x+\beta(x)\partial_y).$$
We can furthermore assume $\alpha(0)\not=0$ so that the local action is transitive together with $\partial_y$;
moreover, $c\not=0$, otherwise the two vector fields commute, which is excluded in the non linearizable case.
Let us firstly discuss the case of normal form (i.a).
In the case where $\partial_y$ is stabilized by $\mathfrak{aut}(\Pi)$, by using \cite[formula (3)]{BMM} (a PDE system for a vector field to be a symmetry of a projective structure) 
for $v=\alpha\partial_x+(y+\beta)\partial_y$, 
one easily deduce that 
$$v=2(x+\alpha_0)\partial_x + (y+\beta_0)\partial_y\ \ \ \text{and}\ \ \ 
(A,B,C,D)=\left(\frac{\gamma_0}{(x+\alpha_0)^{3/2}},\frac{\delta_0}{(x+\alpha_0)},0,1\right),$$
with $\alpha_0,\beta_0,\gamma_0,\delta_0\in\CC$, $\alpha_0\not=0$.
The second case $v=e^{cy}(\alpha(x)\partial_x+\beta(x)\partial_y)$ is less explicit. The invariance of the projective structure in normal form (i.a) allows us to express everything in terms of $\alpha(x)$ and its derivatives:
$$\beta = \frac{\alpha'-c^2\alpha}{2c}\ \ \ \text{and}\ \ \ 
(A,B,C,D)=\left( \frac{\alpha'''-c^4\alpha'}{4c^3\alpha} , -\frac{3\alpha''+c^4\alpha}{4c^2\alpha} , 0 , 1 \right),$$
and finally yields the following differential equation for $\alpha$
$$c^4(\alpha\alpha''-(\alpha')^2)-3c^2(\alpha\alpha'''-\alpha'\alpha'')+2\alpha\alpha''''+\alpha'\alpha'''-3(\alpha'')^2=0.$$
Once we know the $3$-jet of $\alpha$, then we can deduce all the coefficients by mean of this differential equation.
Mind that we can set $\alpha(0)=1$ so that we get a $4$-parameter family of projective structures, taking into account 
the constant $c$, that can further be normalized to $c=1$ by using the $\CC^*$-action.
Equivalently, the family of projective structures is given by the solutions of the system of differential equations
$$A'=\frac{27cA^2+9AB'-3c(B+c^2)B'+c(4B+c^2)(B+c^2)^2}{6(B+c^2)}$$
$$B''=-\frac{27cA(cA-B')-12(B')^2-9c^2(B+c^2)B'+c^2(4B+c^2)(B+c^2)^2}{6(B+c^2)}$$
and we can recover $\alpha$ and $\beta$ by:
$$\frac{\alpha'}{\alpha}=-\frac{3cA+B'}{B+c^2}\ \ \ \text{and}\ \ \ \beta = \frac{\alpha'-c^2\alpha}{2c}.$$

For normal forms (i.b), the discussion is similar, easier though, and one find $v=-\partial_x+(y+c)\partial_y$, $c\in\CC$,
with projective structure $(\alpha_0e^{-x},0,e^x,0)$.

\subsection*{In the case $\mathfrak{aut}(\Pi)=\mathrm{sl}_2(\mathbb{C})$} We just note that $\partial_y$
must be contained in a $2$-dimensional affine Lie subalgebra, and we are in a particular case of the previous one.

\section{Symmetries of flat projective structures}\label{sec:SymFlat}

We say that a projective structure $\Pi$ defined on $U$ by \ref{projective-structure} is \textit{flat} (or \textit{foliated}) if the geodesics are tangent to a pencil of foliations $\{\mathcal{F}_z: \omega_z = \omega_0 + z \omega_{\infty}\}$, where $\omega_0$ and $\omega_{\infty}$ are $1$-forms on $U$ satisfying $ \omega_0 \wedge \omega_{\infty} \neq 0$. %

Here, we classify those projective structures having simultaneously a flat structure and Lie symmetries. In other words, we describe which models in the list of Theorem \ref{thm:Bryant&al} have a flat structure, and how many. As one can see on \cite[Section 3]{FL}, the flatness condition is equivalent to the existence of a semi-local fibration on the dual surface $(U^*, C_0)$ transverse to $C_0$. In \cite[Theorem 5.1]{FL} we show that a given projective structure,
if not linearizable, has at most $2$ flat structures, see also \cite[Theorem A]{FL1}.

\begin{thm}\label{thm:flat+LieSym}
Let $(U,\Pi)$ be a flat projective structure with Lie symmetries: $\mathfrak{aut}(\Pi)\not=\{0\}$. 
Then, at the neighborhood
of a generic point $p\in U$, the pair $(\Pi,\mathfrak{aut}(\Pi))$ and pencil of geodesic foliations 
$\mathcal F_z:\omega_0+z\omega_\infty$ 
can be reduced by local change of coordinate 
to one of the following normal forms:
\begin{itemize}
\item[(i)] $\mathfrak{aut}(\Pi)=\CC\cdot\partial_y$, $(A,B,C,D)=$
\begin{itemize}
\item[(i.a.1)] $(0,0,1+g',g)$ and $\mathcal F_z:\underbrace{e^y(dx+g(x)dy)}_{\omega_0}+z \underbrace{dy}_{\omega_\infty}$; 
\item[(i.a.2)] $(0,0,g',1)$ and $\mathcal F_z:\underbrace{-(dx+(g(x)+y)dy)}_{\omega_0}+z \underbrace{dy}_{\omega_\infty}$; 
\item[(i.b)] $(0,0,g',0)$ and $\mathcal F_z:\underbrace{dx+g(x)dy}_{\omega_0}+z \underbrace{dy}_{\omega_\infty}$;
\end{itemize}
\item[(ii)] $\mathfrak{aut}(\Pi)=\CC\langle\partial_y,\partial_x+y\partial_y\rangle$ and $(A,B,C,D)=$
\begin{itemize}
\item[(ii.a)] $(\alpha e^{x},\beta,0,e^{-2x})$ with $\alpha,\beta\in\CC$ belonging to the cubic nodal curve
\begin{equation}\label{eq:cubicnodalcurve}
\left\{\begin{matrix}
\CC&\to&\Gamma=\{27 \alpha^2+4\beta^3-12\beta^2+9\beta-2=0\}\subset\CC^2\\
\gamma&\mapsto&(\gamma(2\gamma^2-1),2-3\gamma^2)
\end{matrix}\right.
\end{equation}
and the corresponding flat structure is given by
$$\mathcal F_z:\underbrace{e^x(\gamma y+(2\gamma^2-1)e^x)dx-(y+2\gamma e^x)dy}_{\omega_0}+z \underbrace{(dy-\gamma e^xdx)}_{\omega_\infty}.$$
Here, we exclude the cases $\alpha=0$ 
corresponding to (iii) and (iv) below.
\item[(ii.b.1)] $\left(\frac{1-\lambda^2}{4}e^x,0,e^{-x},0\right)$ with $\lambda\in\CC^*$, and 
$$\mathcal F_z:\underbrace{e^{\lambda x}\left[dy-\left(\frac{1-\lambda}{2}\right)e^xdx\right]}_{\omega_0}+z \underbrace{\left[dy-\left(\frac{1+\lambda}{2}\right)e^xdx\right]}_{\omega_\infty};$$
\item[(ii.b.2)] $(\frac{e^x}{4},0,e^{-x},0)$ and $\mathcal F_z:\underbrace{(1-\frac{x}{2})e^xdx+xdy}_{\omega_0}+z \underbrace{(dy-\frac{1}{2}e^xdx)}_{\omega_\infty}$;
\end{itemize}
\item[(iii)] $\mathfrak{aut}(\Pi)=\CC\langle\partial_y,\partial_x+y\partial_y,y\partial_x+\frac{y^2}{2}\partial_y\rangle$
and $(A,B,C,D)=(0,\frac{1}{2},0,e^{-2x})$;
\item[(iv)] $\mathfrak{aut}(\Pi)=\mathrm{sl}_3(\mathbb{C})$ and $(A,B,C,D)=(0,0,0,0)$.
\end{itemize}
\end{thm}

Case (iii) corresponds to the case (ii.a) with $\gamma=\pm\frac{1}{\sqrt{2}}$; 
the two values of $\gamma$ provide the two flat structures for $\Pi$ in this case.
Case (iv) corresponds to the case (ii.a) with $\gamma=0$; in that case, all 
flat structures are described in \cite[Example 3.3]{FL}.

\begin{lemma}\label{lem:XflatNonLin}
Let $(U,\Pi)$ be a projective structure with Lie symmetry $X=\partial_y$ and flat structure $\mathcal F_{\omega_z}$,
with $\omega_z=\omega_0+z\omega_\infty$. If $(U,\Pi)$ is not linearizable, then
\begin{itemize}
\item the flow $\phi_X^t$ of $X$ must preserve the flat structure,
\item the flow $\phi_X^t$ must preserve at least one foliation of the pencil, say $\mathcal F_{\omega_\infty}$,
\item no element of the pencil $\mathcal F_{\omega_z}$ can coincide with the foliation $\mathcal F_X:\{dx=0\}$,
and after change of coordinate $y\mapsto y+\phi(x)$, we may furthermore assume $\omega_\infty=dy$.
\end{itemize}
In particular, at the neighborhood of a generic point $p\in U$, 
we may furthermore assume $\omega_\infty=dy$ in convenient coordinates.
\end{lemma}

\begin{proof}The vector field induces an action on geodesics, and therefore on the dual space $(U^*,C_0)$;
denote by $\check{X}$ the infinitesimal generator. Let $\mathcal H$ be the transverse fibration corresponding
to the flat structure. If $\mathcal H$ is not invariant by the flow $\phi^t_{\check{X}}$,
then we obtain a $1$-parameter family $\mathcal H_t=(\phi^t_{\check{X}})_*\mathcal H$
and deduce from \cite[Theorem 5.1]{FL} that $\Pi$ is linearizable, contradiction. Therefore,
$\check{X}$ preserves $\mathcal H$ and acts on the space of leaves $\simeq\mathbb P^1_z$.
In particular, it has a fixed point, corresponding to an $X$-invariant foliation in the pencil, 
say $\mathcal F_{\omega_\infty}$.

Assume for contradiction that the foliation $\mathcal F_X$, defined by $dx=0$, 
coincides with one of the $\mathcal F_{\omega_z}$'s; since it is $X$-invariant,
we can assume $z=\infty$.
Therefore, we are in the third case of Lemma \ref{lem:XdualSingular}: 
$\check{X}$ has a curve $\Gamma\subset U^*$ of singular points transversal to $C_0$.
Moreover, $\Gamma$ is $\mathcal H$-invariant and $\check{X}$ defines another transverse
fibration $\mathcal F_{\check{X}}$.
If $\Gamma$ is invariant by $\mathcal F_{\check{X}}$, then Lemma \ref{lem:XgeodFlat}
implies that $(U,\Pi)$ is linearizable, contradiction. Therefore, $\Gamma$ is not 
invariant by $\mathcal F_{\check{X}}$, and in particular, the fibrations $\mathcal F_{\check{X}}$ 
and $\mathcal H$ do not coincide. Consider the tangency set $\Sigma:=\mathrm{tang}(\mathcal F_{\check{X}},\mathcal H)$.
Since $\mathcal H$ is $\check{X}$-invariant, $\Sigma$ must be  $\check{X}$-invariant.
Clearly, $\Sigma$ is not contained in the singular set $\mathrm{sing}(\check{X})=\Gamma$ 
and $\Sigma$ is thus $\mathcal F_{\check{X}}$-invariant. Again, this means that $\Sigma$ is a common 
fiber of $\mathcal F_{\check{X}}$ and $\mathcal H$, and the proof of Lemma \ref{lem:XgeodFlat}
implies that $(U,\Pi)$ is linearizable (see also \cite[Theorem 5.3]{FL}), contradiction.
\end{proof}

\begin{proof}[Proof of Theorem \ref{thm:flat+LieSym}]
Let us start with the case $\mathfrak{aut}(\Pi)=\CC\cdot X$ with $X=\partial_y$, and assume $(U,\Pi)$ not linearizable. 
Then, by Lemma \ref{lem:XflatNonLin}, $X$ preserves the pencil of foliations and acts on the parameter space $z\in\mathbb P^1$
fixing $z=\infty$. More precisely, we can assume $\omega_\infty=dy$ and that the action on the pencil is induced by one of the following vector fields
$$z\partial_z,\ \ \ \partial_z\ \ \ \text{or}\ \ \ 0.$$
In the first case, we must have that $(\phi^t_X)_*\omega_z$ is proportional to $\omega_{e^tz}$ for any $t,z\in\CC$; 
since $\omega_z=\omega_0+z\omega_\infty$ and $\omega_\infty$ is $\phi^t_X$-invariant, we deduce
$$(\phi^t_X)_*\omega_0=e^{-t}\omega_0.$$
This implies that $\omega_0=e^{y}(f(x)dx+g(x)dy)$ for some functions $f,g\in\mathbb C\{x\}$, $f(0)\not=0$.
After taking $\int f(x)dx$ as a new coordinate, we get the normal form
$$\omega_0=e^{y}(dx+g(x)dy).$$
We easily derive the projective structure $\Pi$ by derivating ``$\omega_0/\omega_{\infty}$'':
$$0=\left(e^y\left(\frac{1}{y'}+g\right)\right)'=e^yy'\left(\frac{1}{y'}+g\right)+e^y\left(-\frac{y''}{(y')^2}+g'\right)$$
$$\Rightarrow y''=(1+g')(y')^2+(g)(y')^3,\ \ \text{i.e.}\ (A,B,C,D)=(0,0,1+g',g).$$
If the action is now induced by $\partial_z$, then we get 
$$(\phi^t_X)_*\omega_0=\omega_0+t\omega_\infty$$
which gives $\omega_0=f(x)dx+(g(x)-y)dy$, where again we can normalize $f\equiv1$ which gives 
the projective structure 
$$\omega_0=-(dx+(g(x)+y)dy\ \ \ \text{and}\ \ \ (A,B,C,D)=(0,0,g',1).$$
Finally, when the action is trivial on the parameter space $z$, we get that $\omega_0$ is also invariant,
i.e. of the form $f(x)dx+g(x)dy$; we can again normalize $f\equiv1$ and get
$$\omega_0=dx+g(x)dy\ \ \ \text{and}\ \ \ (A,B,C,D)=(0,0,g',0).$$

Let us now consider the case $\mathfrak{aut}(\Pi)=\CC\langle X,Y\rangle$ with $X=\partial_y$ and $Y=\partial_x+y\partial_y$,
and still assume $(U,\Pi)$ not linearizable. Like before, the Lie algebra preserves the pencil $\mathcal F_z$ and induces
an action on the parameter space of the form
$$(X,Y)\vert_z=(\partial_z,z\partial_z),\ \ \ (0,\lambda z\partial_z),\ \ \ (0,\partial_z),\ \ \ \text{or}\ \ \ (0,0),$$
with $\lambda\in\CC^*$. We note that we cannot normalize $\lambda=1$ by homothecy since $Y$ has to 
satisfy $[X,Y]=X$ in the $(x,y)$-variables; different values of $\lambda$ will correspond to different projective structures.
 
In any case for $(X,Y)\vert_z$, $\mathcal F_\infty$ is fixed, and this means that we can write
$$\omega_\infty=d(y-\gamma e^x)$$
for some $\gamma\in\mathbb C$. Indeed, the invariance by $X$ means that the leaves of $\mathcal F_\infty$ 
are $\partial_y$-translates of the leaf $y=f(x)$ passing through the origin, i.e. we can choose $\omega_\infty=d(y-f(x))$;
then, the invariance by $Y$ gives the special form $f(x)=\gamma e^x$.
Here, we have used Lemma \ref{lem:XflatNonLin} to insure that, maybe passing to a generic point $p\in U$, 
we can assume that $\mathcal F_\infty$ is not vertical at $p$. 

If $(X,Y)\vert_z=(\partial_z,z\partial_z)$, then we can check that 
$$\omega_0=(\alpha e^{2x}+\gamma e^xy)dx+(\beta e^x-y)dy$$
for some constants $\alpha,\beta\in\CC$. This normalization is unique up to 
a change of coordinate of the form $\Phi=(x,ay+be^x)$ preserving the Lie algebra;
this allow to reduce the corresponding projective structure $\Pi$ into the normal form (ii.a)
of Theorem \ref{thm:Bryant&al}, yielding after straightforward computation the formulae 
(ii.a) of Theorem \ref{thm:flat+LieSym}.

If $(X,Y)\vert_z=(0,\lambda z\partial_z)$, then we find
$$\omega_0=e^{\lambda x}\left(\alpha e^{x}dx+\beta e^xdy\right)$$
which gives after normalization
$$\omega_z=e^{\lambda x}\left[dy-\left(\frac{1-\lambda}{2}\right)e^xdx\right]+z\left[dy-\left(\frac{1+\lambda}{2}\right)e^xdx\right]$$
$$\text{and}\ \ \ (A,B,C,D)=\left(\frac{1-\lambda^2}{4}e^x,0,e^{-x},0\right).$$
If $(X,Y)\vert_z=(0,\partial_z)$, then we find
$$\omega_0=(\alpha+\gamma x) e^{x}dx+(\beta -x)dy$$
which gives after normalization
$$\omega_z=(1-\frac{x}{2})e^xdx+xdy+z(dy-\frac{1}{2}e^x)\ \ \ \text{and}\ \ \ (A,B,C,D)=\left(\frac{e^x}{4},0,e^{-x},0\right).$$
Finally, if $(X,Y)\vert_z=(0,0)$, then we find
$$\omega_0=\alpha e^{x}dx+\beta dy$$
which gives after normalization $(A,B,C,D)=(0,1,0,0)$, which is linearizable.
\end{proof}

\begin{remark}Projective structures of Theorem \ref{thm:flat+LieSym} (i) can be put in normal form as in Theorem \ref{thm:Bryant&al}. For instance, in the case (i.a.1), using change(\ref{primeiro-cambio}) , one easily get the following form
$$\omega_z=e^{f(x)y}(dx+dy)+zdy\ \ \ \text{and}\ \ \ (A,B,C,D)=(0,f',1+f',1).$$
The final normalization (\ref{segundo-cambio}) is not so explicit, but turning the other way round, 
we can easily check that a normalized projective structure $(A,B,0,1)$ comes from such a flat structure 
iff it satisfies
$$A^2=-\frac{(4B^2+5B-3B'+1)^2}{3(4B+1)}$$
and in this case, $f(x)$ (and the flat structure) is given by 
$$f'=\frac{1}{2}+\frac{1}{2}\sqrt{-3(4B+1)}.$$
In a very similar way, the projective structure $(A,B,0,1)$ comes from the case (i.a.2) iff
$$A^2=-\frac{(4B^2-3B')^2}{108B}$$
and in this case, $g(x)$ (and the flat structure) is given by 
$$g'=\sqrt{-3B}.$$
Finally, one easily check by similar computations that any normal form (i.b) of Theorem \ref{thm:Bryant&al},
i.e. $(A(x),0,e^x,0)$, is also flat, i.e. comes from (i.b) of Theorem \ref{thm:flat+LieSym}.
\end{remark}

\end{document}